\documentclass[]{article}

\usepackage{graphicx}
\usepackage{amsmath}
\usepackage{amssymb}
\usepackage{amsthm}
\usepackage{pxfonts}
\usepackage{enumerate}
\usepackage{color}
\usepackage{mathdots}
\usepackage{sectsty}
\usepackage[hidelinks]{hyperref}

\sectionfont{\scshape\centering\fontsize{11}{14}\selectfont}
\subsectionfont{\scshape\fontsize{11}{14}\selectfont}
\usepackage{fancyhdr}

\newcommand\shorttitle{A matrix Bougerol identity and the Hua-Pickrell measures}
\newcommand\authors{T. Assiotis}

\newtheorem{thm}{Theorem}[section]

\newtheorem{lem}[thm]{Lemma}

\newtheorem{rmk}[thm]{Remark}
\newtheorem{prop}[thm]{Proposition}

\title{\large \bf A MATRIX BOUGEROL IDENTITY AND THE HUA-PICKRELL MEASURES}
\author{\small THEODOROS ASSIOTIS}
\date{}

\begin{document}

\maketitle

\begin{abstract}
We prove a Hermitian matrix version of Bougerol's identity. Moreover, we construct the Hua-Pickrell measures on Hermitian matrices, as stochastic integrals with respect to a drifting Hermitian Brownian motion and with an integrand involving a conjugation by an independent, matrix analogue of the exponential of a complex Brownian motion with drift.
\end{abstract} 

\section{Introduction}

We begin this introduction, by recalling Bougerol's celebrated identity, first established in \cite{Bougerol} in his study of convolution powers of probabilities on certain solvable groups. Let $\left(\beta_t; t \ge 0\right)$ and $\left(\gamma_t; t \ge 0\right)$ be two independent standard Brownian motions starting from $0$. Then, for \textit{fixed} $t \ge 0$, we have the following equalities in law,
\begin{align}\label{Bougerol1D}
\sinh\left(\beta_t\right)\overset{\textnormal{law}}{=}\int_{0}^{t}e^{\beta_s}d\gamma_s\overset{\textnormal{law}}{=}\gamma\left(\int_{0}^{t}e^{2\beta_s}ds\right).
\end{align}

Moreover, if we denote by $\left(\beta_t^{(-\nu)};t \ge 0\right)$ and $\left(\gamma_t^{(-\mu)};t \ge 0\right)$ two independent standard Brownian motions with drifts $-\nu$ and $-\mu$ respectively, the law of the functional, for $\nu>0$,
\begin{align}\label{functional}
\int_{0}^{\infty}e^{\beta_t^{(-\nu)}}d\gamma_t^{(-\mu)}
\end{align}
has density, with respect to Lebesgue measure, given by,
\begin{align*}
f_{\nu,\mu}(x)=\textnormal{c}_{\nu,\mu}  \frac{e^{-2\mu \arctan(x)}}{\left(1+x^2\right)^{\nu+\frac{1}{2}}}.
\end{align*}
Note that this belongs to the much-studied type IV family of Pearson distributions. Both these statements, have been given simple and quite elegant diffusion theoretic proofs by Marc Yor and co-authors in \cite{AliliDufresneYor} and \cite{Yoretal} respectively (see also Marc Yor's monograph \cite{YorExponential} and the survey \cite{VakeroudisSurvey} for more recent developments). The purpose of this note is to obtain the Hermitian matrix analogues of these results. We will establish these by adapting the strategy in the references above to the matrix setting. The real crux here, is understanding what the right matrix analogue should be.

We should also mention that, Marc Yor had an ongoing program for some time, trying to obtain higher dimensional generalizations of Bougerol's identity and study their ramifications (\cite{Bougerolcommunication}). In the last few years, some interesting progress was made in his joint work with Bertoin and Dufresne (\cite{BertoinYorDufresne}), where a generalization involving a (still) one-dimensional process and its local time was discovered. However, our contribution provides the first truly multi-dimensional extension, moreover making a connection between stochastic analysis and the celebrated Hua-Pickrell measures coming from random matrix theory and harmonic analysis on groups. 

 Before continuing, let us explain a bit further the initial motivation behind the study undertaken here. There is a closely related and equally well-known identity in one dimension, originally proven by Dufresne in \cite{Dufresne}: Consider the functional,
\begin{align*}
a_t^{(-\nu)}=\int_{0}^{t}e^{2\beta_s^{(-\nu)}}ds.
\end{align*}
Then, for $\nu>0$,
\begin{align}\label{Dufresneidentity}
a_{\infty}^{(-\nu)}\overset{\textnormal{law}}{=}\frac{1}{2\xi_{\nu}}
\end{align}
where $\xi_{\nu}$ is a Gamma distributed random variable with density $\frac{1}{\Gamma(\nu)}x^{\nu-1}e^{-x}$. Recently, Rider and Valko in \cite{RiderValko} have proven a matrix version of this result, obtaining in place of an inverse Gamma random variable, the inverse Wishart laws. The present paper grew out of my attempt, to both better understand their result and investigate whether other well known matrix laws can be constructed by this diffusion theoretic approach, or "Dufresne procedure" as referred to in \cite{RiderValko}. We finally note that, the second equality in law in (\ref{Bougerol1D}), obtained by a time-change, that links Bougerol's and Dufresne's (one-dimensional) identities, does not appear to have a matrix counterpart.

In order to proceed to state our results, we first need to introduce the Hermitian analogues of the Pearson distribution, of $\left(e^{\beta_t}; t \ge 0\right)$ and $\left(\sinh\left(\beta_t\right); t \ge 0\right)$.

We consider the following measure, denoted by $\mathsf{M}_{HP}^{s,N}$, on the space $\textbf{H}(N)$, of $N\times N$ Hermitian matrices, with $s$ being a complex parameter such that $\Re(s)>-\frac{1}{2}$,
\begin{align}\label{MatrixHuaPickrellMeasure}
\mathsf{M}_{HP}^{s,N}(d\boldsymbol{X})=\textnormal{const} \times \det\left((\boldsymbol{I}+i\boldsymbol{X})^{-s-N}\right)\det \left((\boldsymbol{I}-i\boldsymbol{X})^{-\bar{s}-N}\right) \times d\boldsymbol{X},
\end{align}
where $d\boldsymbol{X}$ denotes Lebesgue measure on $\textbf{H}(N)$. The restriction $\Re(s)>-\frac{1}{2}$ is so that the measure $\mathsf{M}_{HP}^{s,N}$ can be normalized to a probability measure. Its significance in terms of the stochastic processes we shall consider will also be clarified in Lemma \ref{growth} below.

Looking at the radial part of $\mathsf{M}_{HP}^{s,N}(d\boldsymbol{X})$ we get a probability measure on the Weyl chamber $W^N=\{(x_1,\cdots,x_N)\in \mathbb{R}^N:x_1 \le x_2 \le \cdots \le x_N \}$ of log-gas type, which we will denote by $\mu^{s,N}_{HP}$, and is given explicitly by,
\begin{align}\label{HuaPickrellmeasuresevals}
\mu^{s,N}_{HP}(dx)&= \textnormal{const} \times \Delta^2_N(x)\prod_{j=1}^{N}(1+ix_j)^{-s-N}(1-ix_j)^{-\bar{s}-N}dx_j \nonumber\\
&=\textnormal{const} \times \Delta^2_N(x)\prod_{j=1}^{N}(1+x^2_j)^{-\Re(s)-N}e^{2\Im(s)\arg(1+ix_j)}dx_j,
\end{align}
where $x=\left(x_1,\cdots,x_N\right)$ and $\Delta_N(x)=\prod_{1 \le i< j\le N}^{}(x_j-x_i)$ is the Vandermonde determinant. 

Before introducing our stochastic dynamics, we briefly give some of the history of the measures $\mathsf{M}_{HP}^{s,N}$. They were first introduced by Hua Luogeng in the $50$'s in his monograph \cite{Hua} on harmonic analysis in several complex variables and were later in the $80$'s rediscovered independently by Pickrell \cite{Pickrellmeasure} in the context of Grassmann manifolds. Around the turn of the millennium, they were further studied by Neretin in \cite{Neretin} and Borodin and Olshanski investigated their $N \to \infty$ limits as determinantal point processes in \cite{BorodinOlshanskiErgodic}. The reader is referred to \cite{BorodinOlshanskiErgodic} and the more recent study \cite{BufetovQiu} for more of their truly remarkable properties.

We now move on to the matrix stochastic processes we will be dealing with. First some notation. We will denote by $\boldsymbol{A}^{\dagger}$ the complex conjugate of a matrix $\boldsymbol{A}$ and in case it is invertible we write $\boldsymbol{A}^{-\dagger}$ for $\left(\boldsymbol{A}^{\dagger}\right)^{-1}$ and also write $\textnormal{Tr}(\boldsymbol{A})$ for the trace of $\boldsymbol{A}$. Throughout this paper, $\left(\boldsymbol{W}_t;t \ge0\right)$ will be an $N\times N$ complex Brownian matrix. More precisely, its entries consist of independent (scalar) complex Brownian motions.

 We will denote by $\left(\boldsymbol{M}^{(\nu)}_t;t \ge 0\right)$ the matrix analogue of the exponential of complex Brownian motion with drift $\nu$ (the choice of the diffusivity constant is dictated once we fix the normalization of the equation (\ref{matrixhuapickrelldiffusion}) below), given by the solution to the following matrix Stochastic Differential Equation (SDE), starting from $\boldsymbol{M}^{(\nu)}_0=\boldsymbol{I}$,
\begin{align*}
d\boldsymbol{M}^{(\nu)}_t&=\frac{1}{\sqrt{2}}\boldsymbol{M}^{(\nu)}_td\boldsymbol{W}_t+\nu \boldsymbol{M}^{(\nu)}_tdt.
\end{align*}

Moreover, consider the following matrix SDE taking values in $\textbf{H}(N)$ (if $\boldsymbol{X}_0 \in \textbf{H}(N)$), where $\left(\boldsymbol{\Gamma}_t;t \ge 0\right)$ denotes a complex Brownian matrix,

\begin{align}\label{matrixhuapickrelldiffusion}
d\boldsymbol{X}_t=d\boldsymbol{\Gamma}_t\sqrt{\frac{\boldsymbol{I}+\boldsymbol{X}_t^2}{2}}+\sqrt{\frac{\boldsymbol{I}+\boldsymbol{X}_t^2}{2}}d\boldsymbol{\Gamma}_t^{\dagger}+\left[(-N-2\Re(s))\boldsymbol{X}_t+2\Im(s)\boldsymbol{I}+ \textnormal{Tr}\left(\boldsymbol{X}_t\right)\boldsymbol{I}\right]dt.
\end{align}
This is a Hermitian analogue of (a general version of)  $\sinh(\beta_t)$. To see the analogy more clearly, note that, 
\begin{align*}
d\sinh(\beta_t)=\left(1+\sinh^2(\beta_t)\right)^{\frac{1}{2}}d\beta_t+\frac{1}{2}\sinh(\beta_t)dt.
\end{align*}
Hence, to arrive at (\ref{matrixhuapickrelldiffusion}) we simply replaced the scalar (quadratic, with no real roots) diffusion and (linear) drift coefficients by their (symmetrized) matrix analogues. The appearance of the trace drift term is natural and can partly be explained by the calculations required in Propositions \ref{explicitsolution} and \ref{uniqueinvariant} below. Moreover, our choice of both drift and diffusivity constants, is so that (\ref{matrixhuapickrelldiffusion}) has both $\mathsf{M}_{HP}^{s,N}$ as its unique invariant measure and its eigenvalue evolution satisfies a stochastic equation with a certain normalization; this is made precise in Proposition \ref{uniqueinvariant} and its proof. 

One final piece of notation; we will write throughout $\left(\boldsymbol{B}^{(\mu)}_t;t \ge 0\right)$ for a drifting complex Brownian matrix with drift $\mu \in \mathbb{R}$, given by,
\begin{align*}
\boldsymbol{B}^{(\mu)}_t=\boldsymbol{B}_t+\mu \boldsymbol{I}t
\end{align*}
for a complex Brownian matrix $\left(\boldsymbol{B}_t;t \ge 0\right)$ which is \textit{independent} of $\left(\boldsymbol{W}_t;t \ge 0\right)$.

We are now ready to state our two main results. First, the law of the Hermitian analogue of the functional (\ref{functional}), is given by the Hua-Pickrell measure $\mathsf{M}_{HP}^{s,N}$.

\begin{thm}\label{FunctionalHuaPickrell}
Let $\Re(s)>-\frac{1}{2}$. With $\nu=\Re(s)+\frac{N}{2}, \mu=\sqrt{2}\Im(s)$, then,
\begin{align}
\int_{0}^{\infty}\boldsymbol{M}^{\left(-\nu\right)}_t\left(\frac{d\boldsymbol{B}^{(\mu)}_t+d\left(\boldsymbol{B}^{(\mu)}_t\right)^{\dagger}}{\sqrt{2}}\right)\left(\boldsymbol{M}^{\left(-\nu\right)}_t\right)^{\dagger}
\end{align}
is distributed as $\mathsf{M}_{HP}^{s,N}$.
\end{thm}

\begin{rmk}
Comparing with \cite{RiderValko}, the matrix analogue of Dufresne's identity is given by,
\begin{align*}
\int_{0}^{\infty}\boldsymbol{M}^{\left(-\nu\right)}_tdt\left(\boldsymbol{M}^{\left(-\nu\right)}_t\right)^{\dagger}
\end{align*}
which is distributed as an inverse Wishart random matrix. To obtain the Hua-Pickrell measures, we have replaced the $dt$ integration by a stochastic integral with respect to an independent (drifting) Hermitian Brownian motion, $\left(\boldsymbol{B}^{(\mu)}_t+\left(\boldsymbol{B}^{(\mu)}_t\right)^{\dagger};t \ge 0\right)$.
\end{rmk}

Finally, we have the following Hermitian version of Bougerol's identity (\ref{Bougerol1D}).

\begin{thm}\label{MatrixBougerolTheorem}
With $\nu=\Re(s)+\frac{N}{2}, \mu=\sqrt{2}\Im(s)$, denote by $\tilde{\boldsymbol{X}}_t^{\mu, \nu}$ the unique solution of (\ref{matrixhuapickrelldiffusion}) starting from the $\boldsymbol{0}$ matrix. Then, for fixed $t>0$,
\begin{align}
\tilde{\boldsymbol{X}}_t^{\mu, \nu}\overset{law}{=}\int_{0}^{t}\boldsymbol{M}^{\left(-\nu\right)}_u\left(\frac{d\boldsymbol{B}^{(\mu)}_u+d\left(\boldsymbol{B}^{(\mu)}_u\right)^{\dagger}}{\sqrt{2}}\right)\left(\boldsymbol{M}^{\left(-\nu\right)}_u\right)^{\dagger}.
\end{align}
\end{thm}

\paragraph{Acknowledgements} I would like to thank Jon Warren for some very useful conversations. Financial support from EPSRC through the MASDOC DTC grant number EP/HO23364/1 is gratefully acknowledged.

\section{Preliminaries, Auxiliary results and Proofs of Theorems}

As in the introduction, we denote by $\left(\boldsymbol{M}^{(\nu)}_t;t \ge 0\right)$ the matrix analogue of the exponential of complex Brownian motion with drift $\nu$ (and diffusivity $\frac{1}{\sqrt{2}}$), starting from $\boldsymbol{M}^{(\nu)}_0=\boldsymbol{I}$,
\begin{align*}
d\boldsymbol{M}^{(\nu)}_t&=\frac{1}{\sqrt{2}}\boldsymbol{M}^{(\nu)}_td\boldsymbol{W}_t+\nu \boldsymbol{M}^{(\nu)}_tdt,\\
d\left(\boldsymbol{M}^{(\nu)}_t\right)^{\dagger}&=\frac{1}{\sqrt{2}}d\boldsymbol{W}_t^{\dagger}\left(\boldsymbol{M}^{(\nu)}_t\right)^{\dagger}+\nu \left(\boldsymbol{M}^{(\nu)}_t\right)^{\dagger}dt.
\end{align*}
A simple application of It\^{o}'s formula gives the following SDE for $\left(\det\left(\boldsymbol{M}^{(\nu)}_t\right);t \ge 0\right)$,
\begin{align*}
d\det\left(\boldsymbol{M}^{(\nu)}_t\right)=\det\left(\boldsymbol{M}^{(\nu)}_t\right)\left(\frac{1}{\sqrt{2}}\textnormal{tr}\left(d\boldsymbol{W}_t\right)+\nu N  dt\right).
\end{align*}
Solving it, we get,
\begin{align*}
\det\left(\boldsymbol{M}^{(\nu)}_t\right)=\exp\left(\frac{1}{\sqrt{2}}\textnormal{tr}\left(\boldsymbol{W}_t\right)+\nu N t\right).
\end{align*}
Thus, $\left(\boldsymbol{M}^{(\nu)}_t;t \ge 0\right)$ is almost surely invertible. Moreover, by applying It\^{o}'s formula to the identity $\boldsymbol{M}^{(\nu)}_t\left(\boldsymbol{M}^{(\nu)}_t\right)^{-1}=\boldsymbol{I}$, we easily obtain the following description of the dynamics of its inverse $\left(\left(\boldsymbol{M}^{(\nu)}_t\right)^{-1};t \ge 0\right)$,
\begin{align*}
d\left(\boldsymbol{M}^{(\nu)}_t\right)^{-1}&=-\frac{1}{\sqrt{2}}d\boldsymbol{W}_t\left(\boldsymbol{M}^{(\nu)}_t\right)^{-1}-\nu \left(\boldsymbol{M}^{(\nu)}_t\right)^{-1}dt,\\
d\left(\boldsymbol{M}^{(\nu)}_t\right)^{-\dagger}&=-\frac{1}{\sqrt{2}}\left(\boldsymbol{M}^{(\nu)}_t\right)^{-\dagger}d\boldsymbol{W}_t^{\dagger}-\nu \left(\boldsymbol{M}^{(\nu)}_t\right)^{-\dagger}dt.
\end{align*}

We will also need the notion and a precise description of the evolution of the \textit{time-reversal} of $\left(\boldsymbol{M}^{(\nu)}_t;t\ge 0\right)$. For $T\ge 0 $ fixed, we will denote this time-reversed process by $\left(\boldsymbol{N}^{(\nu)}_t;0 \le t \le T\right)=\left(\left(\boldsymbol{M}^{(\nu)}_T\right)^{-1}\boldsymbol{M}^{(\nu)}_{T-t};0 \le t \le T\right)$. Then, we have the following lemma.
\begin{lem}\label{timereversal}
$\left(\boldsymbol{N}^{(\nu)}_t;0 \le t \le T\right)$ satisfies,
 \begin{align*}
 d\boldsymbol{N}^{(\nu)}_t=\frac{1}{\sqrt{2}}\boldsymbol{N}^{(\nu)}_td\tilde{\boldsymbol{W}}_t-\nu\boldsymbol{N}^{(\nu)}_tdt,
 \end{align*}
 for a complex Brownian matrix $\tilde{\boldsymbol{W}}$. In particular, it is distributed as $\left(\boldsymbol{M}^{(-\nu)}_t;0 \le t \le T\right)$ starting from $\boldsymbol{I}$.
\end{lem}

Furthermore, we have the following result for the rate of growth of $\left(\boldsymbol{M}^{(-\nu)}_t;t \ge 0\right)$ as $t \to \infty$; this ensures the convergence of the various matrix integrals we have encountered under the assumption $\Re(s)>-\frac{1}{2}$.
\begin{lem}\label{growth}
Let $\left(\eta^{(-\nu)}_1(t) \le \cdots \le \eta^{(-\nu)}_N(t);t \ge 0 \right)$ denote the squared singular values of $\left(\boldsymbol{M}^{(-\nu)}_t;t \ge 0\right)$. Then, almost surely,
\begin{align*}
\lim_{t \to \infty}\frac{1}{t}\log \eta_N^{(-\nu)}(t) \le -2\nu +N-1.
\end{align*}
In particular, if $\nu=\Re(s)+\frac{N}{2}$ for $\Re(s)>-\frac{1}{2}$ we have,
\begin{align*}
\lim_{t \to \infty}\frac{1}{t}\log \eta_N^{(-\nu)}(t) <0
\end{align*}
and hence, for any matrix norm $\| \cdot \|$ we have,
\begin{align*}
\lim_{t \to \infty}\frac{1}{t}\log\left\| \boldsymbol{M}^{(-\nu)}_t \right\| <0 , \textnormal{ almost surely}.
\end{align*}
\end{lem}

It is a remarkable fact, that the solution of (\ref{matrixhuapickrelldiffusion}), for any initial condition $\boldsymbol{X}_0$, can be written out explicitly:

\begin{prop}\label{explicitsolution}
With $\nu=\Re(s)+\frac{N}{2}, \mu=\sqrt{2}\Im(s)$, then the unique strong solution of (\ref{matrixhuapickrelldiffusion}), starting from $\boldsymbol{X}_0 \in \textbf{H}(N)$ is given explicitly by,
\begin{align}
\left(\boldsymbol{M}^{(\nu)}_t\right)^{-1}\left[\boldsymbol{X}_0+\int_{0}^{t}\boldsymbol{M}^{\left(\nu\right)}_u\left(\frac{d\boldsymbol{B}^{(\mu)}_u+d\left(\boldsymbol{B}^{(\mu)}_u\right)^{\dagger}}{\sqrt{2}}\right)\left(\boldsymbol{M}^{\left(\nu\right)}_u\right)^{\dagger}\right]\left(\boldsymbol{M}^{(\nu)}_t\right)^{-\dagger}.
\end{align}
\end{prop}

The final ingredient that we will make use of is the following.

\begin{prop}\label{uniqueinvariant}
Let $\Re(s)>- \frac{1}{2}$. Then, the unique strong solution $\left(\boldsymbol{X}_t;t \ge 0\right)$ to (\ref{matrixhuapickrelldiffusion}) has $\mathsf{M}_{HP}^{s,N}$ as its unique invariant measure.
\end{prop}

We are now in position to quickly prove our two main results.

\begin{proof}[Proof of Theorem \ref{MatrixBougerolTheorem}]
This follows immediately from Proposition \ref{explicitsolution}, by making the change of variables $u\mapsto t-u$, using the time-reversal Lemma \ref{timereversal} for $\left(\left(\boldsymbol{M}^{(\nu)}_t\right)^{-1}\boldsymbol{M}^{\left(\nu\right)}_{t-u};0 \le u \le t\right)$ and finally noting invariance under time-reversal of the matrix Brownian motion $\boldsymbol{B}$.
\end{proof}

\begin{proof}[Proof of Theorem \ref{FunctionalHuaPickrell}]
From Theorem \ref{MatrixBougerolTheorem} we have that,
\begin{align}
\tilde{\boldsymbol{X}}_t^{\mu, \nu}\overset{law}{=}\int_{0}^{t}\boldsymbol{M}^{\left(-\nu\right)}_u\left(\frac{d\boldsymbol{B}^{(\mu)}_u+d\left(\boldsymbol{B}^{(\mu)}_u\right)^{\dagger}}{\sqrt{2}}\right)\left(\boldsymbol{M}^{\left(-\nu\right)}_u\right)^{\dagger}.
\end{align}
Moreover, by Lemma \ref{growth} we have that for $\Re(s)>-\frac{1}{2}$ almost surely,
\begin{align*}
\int_{0}^{t}\boldsymbol{M}^{\left(-\nu\right)}_u\left(\frac{d\boldsymbol{B}^{(\mu)}_u+d\left(\boldsymbol{B}^{(\mu)}_u\right)^{\dagger}}{\sqrt{2}}\right)\left(\boldsymbol{M}^{\left(-\nu\right)}_u\right)^{\dagger} \underset{t \to \infty}{\longrightarrow} \int_{0}^{\infty}\boldsymbol{M}^{\left(-\nu\right)}_u\left(\frac{d\boldsymbol{B}^{(\mu)}_u+d\left(\boldsymbol{B}^{(\mu)}_u\right)^{\dagger}}{\sqrt{2}}\right)\left(\boldsymbol{M}^{\left(-\nu\right)}_u\right)^{\dagger}.
\end{align*}
Thus, we have the following convergence in law (in fact for any initial condition $\boldsymbol{X}_0 \in \textbf{H}(N)$ and not just for the $\boldsymbol{0}$ matrix),
\begin{align*}
\tilde{\boldsymbol{X}}_t^{\mu, \nu}\underset{t \to \infty}{\overset{\textnormal{law}}{\longrightarrow}}\int_{0}^{\infty}\boldsymbol{M}^{\left(-\nu\right)}_u\left(\frac{d\boldsymbol{B}^{(\mu)}_u+d\left(\boldsymbol{B}^{(\mu)}_u\right)^{\dagger}}{\sqrt{2}}\right)\left(\boldsymbol{M}^{\left(-\nu\right)}_u\right)^{\dagger}.
\end{align*}
But by Proposition \ref{uniqueinvariant}, for $\Re(s)>-\frac{1}{2}$, $\mathsf{M}_{HP}^{s,N}$ is the unique invariant probability measure of (\ref{matrixhuapickrelldiffusion}) and so,
\begin{align*}
\int_{0}^{\infty}\boldsymbol{M}^{\left(-\nu\right)}_u\left(\frac{d\boldsymbol{B}^{(\mu)}_u+d\left(\boldsymbol{B}^{(\mu)}_u\right)^{\dagger}}{\sqrt{2}}\right)\left(\boldsymbol{M}^{\left(-\nu\right)}_u\right)^{\dagger} \textnormal{is distributed as } \mathsf{M}_{HP}^{s,N}.
\end{align*}
\end{proof}
\section{Proofs of auxiliary results}

\begin{proof}[Proof of Proposition \ref{explicitsolution}]
The fact that (\ref{matrixhuapickrelldiffusion}) has a unique strong solution has been proven in Section 8 of \cite{HuaPickrell} (by a standard argument found also in \cite{Wishart} and \cite{Doumerc} for example). It suffices to check that,
\begin{align*}
\left(\boldsymbol{M}^{(\nu)}_t\right)^{-1}\left[\boldsymbol{X}_0+\int_{0}^{t}\boldsymbol{M}^{\left(\nu\right)}_u\left(\frac{d\boldsymbol{B}^{(\mu)}_u+d\left(\boldsymbol{B}^{(\mu)}_u\right)^{\dagger}}{\sqrt{2}}\right)\left(\boldsymbol{M}^{\left(\nu\right)}_u\right)^{\dagger}\right]\left(\boldsymbol{M}^{(\nu)}_t\right)^{-\dagger}
\end{align*}
indeed solves (\ref{matrixhuapickrelldiffusion}) for $\nu=\Re(s)+\frac{N}{2}, \mu=\sqrt{2}\Im(s)$. The initial condition is immediate and in order to ease notation, we will suppress any dependence on it in what follows. Let $\tilde{\boldsymbol{X}}_t^{\mu, \nu}$ denote the expression above. Then, applying It\^{o}'s formula we get,
\begin{align*}
d\tilde{\boldsymbol{X}}_t^{\mu, \nu}&=d\left(\left(\boldsymbol{M}^{(\nu)}_t\right)^{-1}\right)\left[\boldsymbol{X}_0+\int_{0}^{t}\boldsymbol{M}^{\left(\nu\right)}_u\left(\frac{d\boldsymbol{B}^{(\mu)}_u+d\left(\boldsymbol{B}^{(\mu)}_u\right)^{\dagger}}{\sqrt{2}}\right)\left(\boldsymbol{M}^{\left(\nu\right)}_u\right)^{\dagger}\right]\left(\boldsymbol{M}^{(\nu)}_t\right)^{-\dagger}\\
&+\left(\boldsymbol{M}^{(\nu)}_t\right)^{-1}\left[\boldsymbol{X}_0+\int_{0}^{t}\boldsymbol{M}^{\left(\nu\right)}_u\left(\frac{d\boldsymbol{B}^{(\mu)}_u+d\left(\boldsymbol{B}^{(\mu)}_u\right)^{\dagger}}{\sqrt{2}}\right)\left(\boldsymbol{M}^{\left(\nu\right)}_u\right)^{\dagger}\right]d\left(\left(\boldsymbol{M}^{(\nu)}_t\right)^{-\dagger}\right)\\
&+\left(\boldsymbol{M}^{(\nu)}_t\right)^{-1}\left[\boldsymbol{M}^{\left(\nu\right)}_t\left(\frac{d\boldsymbol{B}^{(\mu)}_t+d\left(\boldsymbol{B}^{(\mu)}_t\right)^{\dagger}}{\sqrt{2}}\right)\left(\boldsymbol{M}^{\left(\nu\right)}_t\right)^{\dagger}\right]\left(\boldsymbol{M}^{(\nu)}_t\right)^{-\dagger}\\
&+d\left(\left(\boldsymbol{M}^{(\nu)}_t\right)^{-1}\right)\left[\boldsymbol{X}_0+\int_{0}^{t}\boldsymbol{M}^{\left(\nu\right)}_u\left(\frac{d\boldsymbol{B}^{(\mu)}_u+d\left(\boldsymbol{B}^{(\mu)}_u\right)^{\dagger}}{\sqrt{2}}\right)\left(\boldsymbol{M}^{\left(\nu\right)}_u\right)^{\dagger}\right]d\left(\left(\boldsymbol{M}^{(\nu)}_t\right)^{-\dagger}\right).
\end{align*}
Note that, the terms of the form,
\begin{align*}
d\left(\left(\boldsymbol{M}^{(\nu)}_t\right)^{-1}\right)\left[\boldsymbol{M}^{\left(\nu\right)}_t\left(\frac{d\boldsymbol{B}^{(\mu)}_t+d\left(\boldsymbol{B}^{(\mu)}_t\right)^{\dagger}}{\sqrt{2}}\right)\left(\boldsymbol{M}^{\left(\nu\right)}_t\right)^{\dagger}\right]\left(\boldsymbol{M}^{(\nu)}_t\right)^{-\dagger}=0,\\
\left(\boldsymbol{M}^{(\nu)}_t\right)^{-1}\left[\boldsymbol{M}^{\left(\nu\right)}_t\left(\frac{d\boldsymbol{B}^{(\mu)}_t+d\left(\boldsymbol{B}^{(\mu)}_t\right)^{\dagger}}{\sqrt{2}}\right)\left(\boldsymbol{M}^{\left(\nu\right)}_t\right)^{\dagger}\right]d\left(\left(\boldsymbol{M}^{(\nu)}_t\right)^{-\dagger}\right)=0,
\end{align*}
by independence of $\boldsymbol{B}$ and the driving Brownian motion $\boldsymbol{W}$ of $\boldsymbol{M}^{(\nu)}$. Moreover, using the fact that for a (scalar) complex Brownian motion $\beta$ we have the following quadratic covariation rules: $d\beta d \beta =0, d\beta d \bar{\beta}=2$; we easily obtain (we will do a similar and more complicated calculation below) for a matrix $\boldsymbol{A}$ and matricial complex Browian motion $\boldsymbol{W}$,
\begin{align*}
d\boldsymbol{W}_t \boldsymbol{A} d\boldsymbol{W}^{\dagger}_t=2\textnormal{Tr}(\boldsymbol{A})\boldsymbol{I}dt.
\end{align*}
Hence,
\begin{align*}
&d\tilde{\boldsymbol{X}}_t^{\mu, \nu}=-\frac{1}{\sqrt{2}}d\boldsymbol{W}_t\tilde{\boldsymbol{X}}_t^{\mu, \nu}+\frac{d\boldsymbol{B}_t}{\sqrt{2}}-\frac{1}{\sqrt{2}}\tilde{\boldsymbol{X}}_t^{\mu, \nu}d\boldsymbol{W}^{\dagger}_t+\frac{d\boldsymbol{B}^{\dagger}_{t}}{\sqrt{2}}-2\nu \tilde{\boldsymbol{X}}_t^{\mu, \nu}dt+\frac{2}{\sqrt{2}}\mu \boldsymbol{I} dt+\textnormal{Tr}(\tilde{\boldsymbol{X}}_t^{\mu, \nu})\boldsymbol{I}dt.\\
&=\left[-\frac{1}{\sqrt{2}}d\boldsymbol{W}_t\tilde{\boldsymbol{X}}_t^{\mu, \nu}+\frac{d\boldsymbol{B}_t}{\sqrt{2}}\right]\left(\frac{\boldsymbol{I}+\left(\tilde{\boldsymbol{X}}_t^{\mu, \nu}\right)^2}{2}\right)^{-\frac{1}{2}}\sqrt{\frac{\boldsymbol{I}+\left(\tilde{\boldsymbol{X}}_t^{\mu, \nu}\right)^2}{2}}+\sqrt{\frac{\boldsymbol{I}+\left(\tilde{\boldsymbol{X}}_t^{\mu, \nu}\right)^2}{2}}\left(\frac{\boldsymbol{I}+\left(\tilde{\boldsymbol{X}}_t^{\mu, \nu}\right)^2}{2}\right)^{-\frac{1}{2}}\left[-\frac{1}{\sqrt{2}}d\boldsymbol{W}_t\tilde{\boldsymbol{X}}_t^{\mu, \nu}+\frac{d\boldsymbol{B}_t}{\sqrt{2}}\right]\\
&-2\nu \tilde{\boldsymbol{X}}_t^{\mu, \nu}dt+\frac{2}{\sqrt{2}}\mu \boldsymbol{I} dt+\textnormal{Tr}(\tilde{\boldsymbol{X}}_t^{\mu, \nu})\boldsymbol{I}dt.
\end{align*}
Writing, $d\boldsymbol{\Gamma}_t=\left[-\frac{1}{\sqrt{2}}d\boldsymbol{W}_t\tilde{\boldsymbol{X}}_t^{\mu, \nu}+\frac{d\boldsymbol{B}_t}{\sqrt{2}}\right]\left(\frac{\boldsymbol{I}+\left(\tilde{\boldsymbol{X}}_t^{\mu, \nu}\right)^2}{2}\right)^{-\frac{1}{2}}$ and then using Levy's characterization and $(d\boldsymbol{\Gamma}_t)_{ij}(d\boldsymbol{\Gamma}_t)_{i'j'}=0$, $(d\boldsymbol{\Gamma}_t)_{ij}(\bar{d\boldsymbol{\Gamma}_t)}_{i'j'}=2\delta_{i,i'}\delta_{j,j'}dt$ we deduce that $\left(\boldsymbol{\Gamma}_t;t \ge 0\right)$ is a complex Brownian matrix. The fact that $(d\boldsymbol{\Gamma}_t)_{ij}(d\boldsymbol{\Gamma}_t)_{i'j'}=0$ is immediate; to check $(d\boldsymbol{\Gamma}_t)_{ij}(\bar{d\boldsymbol{\Gamma}_t)}_{i'j'}=2\delta_{i,i'}\delta_{j,j'}dt$, writing $\boldsymbol{Y}_t=\left(\boldsymbol{I}+\left(\tilde{\boldsymbol{X}}_t^{\mu, \nu}\right)^2\right)^{-\frac{1}{2}}$ we have,
\begin{align*}
(d\boldsymbol{\Gamma}_t)_{ij}(\bar{d\boldsymbol{\Gamma}_t)}_{i'j'}&=\left(\sum_{k,l}^{}-d\boldsymbol{W}^{i,k}_t\left(\tilde{\boldsymbol{X}}_t^{\mu, \nu}\right)^{kl}\boldsymbol{Y}^{lj}_t+\sum_{l}^{}d\boldsymbol{B}_t^{il}\boldsymbol{Y}_t^{lj}\right)\left(\sum_{k',l'}^{}-d\bar{\boldsymbol{W}}^{i',k'}_t\overline{\left(\tilde{\boldsymbol{X}}_t^{\mu, \nu}\right)^{k'l'}}\bar{\boldsymbol{Y}}_t^{l'j'}+\sum_{l'}^{}d\bar{\boldsymbol{B}}_t^{i'l'}\bar{\boldsymbol{Y}}_t^{l'j'}\right)\\
&=\left[2 \delta_{i,i'}\sum_{k,k',l,l'}^{}\delta_{k,k'}\left(\tilde{\boldsymbol{X}}_t^{\mu, \nu}\right)^{kl}\boldsymbol{Y}^{lj}_t\overline{\left(\tilde{\boldsymbol{X}}_t^{\mu, \nu}\right)^{k'l'}}\bar{\boldsymbol{Y}}_t^{l'j'}+2\delta_{i,i'}\sum_{l,l'}^{}\delta_{l,l'}\boldsymbol{Y}_t^{lj}\bar{\boldsymbol{Y}}_t^{l'j'}\right]dt\\
&=\left[2 \delta_{i,i'}\sum_{l,l'}^{}\left[\sum_{k}^{}\left(\tilde{\boldsymbol{X}}_t^{\mu, \nu}\right)^{kl}\overline{\left(\tilde{\boldsymbol{X}}_t^{\mu, \nu}\right)^{kl'}}\right]\boldsymbol{Y}^{lj}_t\bar{\boldsymbol{Y}}_t^{l'j'}+2\delta_{i,i'}\sum_{l,l'}^{}\delta_{l,l'}\boldsymbol{Y}_t^{lj}\bar{\boldsymbol{Y}}_t^{l'j'}\right]dt\\
&=\left[2 \delta_{i,i'}\sum_{l,l'}^{}\left[\sum_{k}^{}\left(\tilde{\boldsymbol{X}}_t^{\mu, \nu}\right)^{kl}\left(\tilde{\boldsymbol{X}}_t^{\mu, \nu}\right)^{l'k}\right]\boldsymbol{Y}^{lj}_t\boldsymbol{Y}^{j'l'}_t+2\delta_{i,i'}\sum_{l,l'}^{}\delta_{l,l'}\boldsymbol{Y}_t^{lj}\boldsymbol{Y}_t^{j'l'}\right]dt\\
&=\left[2 \delta_{i,i'}\sum_{l,l'}^{}\boldsymbol{Y}^{j'l'}_t\left[\left(\tilde{\boldsymbol{X}}_t^{\mu, \nu}\right)^2_{l'l}+\boldsymbol{I}_{l'l}\right]\boldsymbol{Y}^{lj}_t\right]dt\\
&=\left[2 \delta_{i,i'}\left(\boldsymbol{Y}_t\left(\boldsymbol{I}+\left(\tilde{\boldsymbol{X}}_t^{\mu, \nu}\right)^2\right)\boldsymbol{Y}_t\right)_{j'j}\right]dt=2 \delta_{i,i'}\delta_{j,j'}dt,
\end{align*}
where we have used the fact that both $\tilde{\boldsymbol{X}}_t^{\mu, \nu}$ and $\boldsymbol{Y}_t$ are Hermitian in the fourth equality. Thus,
\begin{align*}
d\tilde{\boldsymbol{X}}_t^{\mu, \nu}=d\boldsymbol{\Gamma}_t\sqrt{\frac{\boldsymbol{I}+\left(\tilde{\boldsymbol{X}}_t^{\mu, \nu}\right)^2}{2}}+\sqrt{\frac{\boldsymbol{I}+\left(\tilde{\boldsymbol{X}}_t^{\mu, \nu}\right)^2}{2}}d\boldsymbol{\Gamma}^{\dagger}_t-2\nu \tilde{\boldsymbol{X}}_t^{\mu, \nu}dt+\frac{2}{\sqrt{2}}\mu \boldsymbol{I} dt+\textnormal{Tr}(\tilde{\boldsymbol{X}}_t^{\mu, \nu})\boldsymbol{I}dt.
\end{align*}
Finally, to match with (\ref{matrixhuapickrelldiffusion}), we just need to take $\nu=\Re(s)+\frac{N}{2},\mu=\sqrt{2}\Im(s)$.
\end{proof}

\begin{proof}[Proof of Proposition \ref{uniqueinvariant}]
This has already been observed in Section 8 of \cite{HuaPickrell}. The argument goes as follows. Let $\mathbb{U}(N)$ denote the $N \times N$ unitary group. Then, by $\mathbb{U}(N)$-invariance of the law of the dynamics of (\ref{matrixhuapickrelldiffusion}) (invariance under conjugation, $x \mapsto U^{\dagger}xU$, for $U \in \mathbb{U}(N)$), it suffices to show that its spectral evolution, denoted by $\left(x_1(t),\cdots x_N(t);t \ge 0\right)$ has $\mu^{s,N}_{HP}(dx)$ as its unique invariant probability measure. Using Theorem 4 of \cite{MatrixYamadaWatanabe} for example, we obtain that $\left(x_1(t),\cdots x_N(t);t \ge 0\right)$ follows the stochastic differential system,
\begin{align*}
dx_i(t)=\sqrt{2(1+x^2_i(t))}d\beta_i(t)+\left(2\Im(s)+\left(2-2N-2\Re(s)\right)x_i(t)+\sum_{j\ne i}^{}\frac{2\left(1+x^2_i(t)\right)}{x_i(t)-x_j(t)}\right)dt, \ 1\le i \le N,
\end{align*}
for some independent standard (real) Brownian motions $\{\beta_i\}_{i=1}^N$. It was proven in Lemma 4.3 of \cite{HuaPickrell}, using the general results of \cite{Graczyk}, that this system of SDEs has a unique strong solution, with no explosions or collisions, even if started from a degenerate point (when $x_i(0)=x_j(0)$ for $i\ne j$). Let $\left(P_{HP}^{s,N}(t); t \ge 0\right)$ denote the Markov semigroup associated with it. Then, checking invariance $\mu^{s,N}_{HP}P_{HP}^{s,N}(t)=\mu^{s,N}_{HP}, t\ge0 $ is particularly simple, since the argument becomes essentially one-dimensional. This is because the kernel, $P_{HP}^{s,N}(t)(x,dy)$ in $W^N$, of the semigroup $\left(P_{HP}^{s,N}(t); t \ge 0\right)$ has a determinantal structure, given by an \textit{h-transform} of a \textit{Karlin-McGregor} semigroup. Namely,
\begin{align*}
P_{HP}^{s,N}(t)(x,dy)=e^{-ct}\frac{\Delta_N(x)}{\Delta_N(y)}\det\left(p_t^{s,N}(x_i,y_j)\right)^N_{i,j=1}dy_1\cdots dy_N
\end{align*}
where $x=(x_1,\cdots,x_N)$, $y=(y_1,\cdots,y_N)$, $p_t^{s,N}(z,w)$ is the strictly positive transition density, with respect to Lebesgue measure in $\mathbb{R} $, of the one-dimensional diffusion process with generator,
\begin{align*}
L^{(N)}_s=(w^2+1)\frac{d^2}{dw^2}+\left[\left(2-2N-2\Re(s)\right)w+2\Im(s)\right]\frac{d}{dw},
\end{align*}
which is furthermore, reversible with respect to the measure,
\begin{align*}
m_s^{(N)}(w)dw=(1+w^2)^{-\Re(s)-N}e^{2\Im(s)\textnormal{arg}(1+iw)}dw
\end{align*}
and finally $c$ is a constant. Invariance and uniqueness of $\mu_{HP}^{s,N}$ then follow easily. The reader is referred to Proposition 4.4 of \cite{HuaPickrell} for the details.

We can alternatively argue for uniqueness of the invariant measure $\mathsf{M}_{HP}^{s,N}(d\boldsymbol{X})$,  by noting that the diffusion matrix of (\ref{matrixhuapickrelldiffusion}) is uniformly positive definite, from which we deduce (see \cite{Stroock} for example) that if $\mathsf{G}_{\boldsymbol{X}}$ denotes the generator of the unique solution of (\ref{matrixhuapickrelldiffusion}), then $\partial_t-\mathsf{G}^*_{\boldsymbol{X}}$ is hypoelliptic.
\end{proof}

\begin{proof}[Proof of Lemma \ref{timereversal}]
Let $T>0$ be fixed. For $0\le t \le T$, we have that,
\begin{align*}
\boldsymbol{M}^{(\nu)}_T=\boldsymbol{M}^{(\nu)}_{T-t}+\int_{T-t}^{T}\boldsymbol{M}^{(\nu)}_{u}d\boldsymbol{W}_u+\nu \int_{T-t}^{T}\boldsymbol{M}^{(\nu)}_{u}du.
\end{align*}
Hence, by multiplying by $\left(\boldsymbol{M}^{(\nu)}_{T}\right)^{-1}$ and making the change of variables $u \mapsto T-u$ in the Lebesgue integral,
\begin{align*}
\boldsymbol{N}^{(\nu)}_t&=\boldsymbol{I}-\int_{T-t}^{T}\left(\boldsymbol{M}^{(\nu)}_{T}\right)^{-1}\boldsymbol{M}^{(\nu)}_{u}d\boldsymbol{W}_u-\nu \int_{T-t}^{T}\left(\boldsymbol{M}^{(\nu)}_{T}\right)^{-1}\boldsymbol{M}^{(\nu)}_{u}du\\
&=\boldsymbol{I}-\int_{T-t}^{T}\left(\boldsymbol{M}^{(\nu)}_{T}\right)^{-1}\boldsymbol{M}^{(\nu)}_{u}d\boldsymbol{W}_u-\nu \int_{0}^{t}\boldsymbol{N}^{(\nu)}_udu.
\end{align*}
Now, to treat the stochastic integral term, begin by writing $\tilde{\boldsymbol{W}}_t=\boldsymbol{W}_{T-t}-\boldsymbol{W}_T$ for the time-reversed Brownian motion. We note that, this is again a Brownian motion with filtration given by,
\begin{align*}
\mathcal{F}^{\tilde{\boldsymbol{W}}}_{r,s}=\sigma\left(\tilde{\boldsymbol{W}}_s-\tilde{\boldsymbol{W}}_r|r \le s \le T\right)=\sigma\left(\boldsymbol{W}_u|T-s \le u \le T-r\right)=\mathcal{F}^{\boldsymbol{W}}_{T-s,T-r}.
\end{align*}
Using an approximation by Riemann sums, see for example Proposition 7.2.11 of \cite{FranchiLeJan} where this is done, we can write the stochastic integral in consideration as an It\^{o} integral with respect to the time-reversed Brownian motion $\tilde{\boldsymbol{W}}$, namely,
\begin{align*}
\int_{T-t}^{T}\left(\boldsymbol{M}^{(\nu)}_{T}\right)^{-1}\boldsymbol{M}^{(\nu)}_{u}d\boldsymbol{W}_u=-\int_{0}^{t}\boldsymbol{N}^{(\nu)}_ud\tilde{\boldsymbol{W}}_u-\int_{0}^{t}d\left\langle\boldsymbol{N}^{(\nu)},\tilde{\boldsymbol{W}}\right\rangle_{u}.
\end{align*}
Observe that, the martingale part of $d\boldsymbol{N}^{(\nu)}$ is $-\boldsymbol{N}^{(\nu)}d\tilde{\boldsymbol{W}}$ and thus,
\begin{align*}
\int_{0}^{t}d\left\langle\boldsymbol{N}^{(\nu)},\tilde{\boldsymbol{W}}\right\rangle_{u}=\int_{0}^{t}\boldsymbol{N}^{(\nu)}_ud\left\langle\tilde{\boldsymbol{W}},\tilde{\boldsymbol{W}}\right\rangle_{u}=0,
\end{align*}
since we are dealing with complex Brownian motions (in case we were working with real Brownian matrices we would have picked up an extra drift term). The result then follows.
\end{proof}

\begin{proof}[Proof of Lemma \ref{growth}]
This is essentially an adaptation of Lemma 11 of \cite{RiderValko}. We consider the following stochastic process $\left(\boldsymbol{Z}_t^{(-\nu)};t \ge 0\right)=\left(\boldsymbol{M}_t^{(-\nu)}\left(\boldsymbol{M}_t^{(-\nu)}\right)^{\dagger};t \ge 0\right)$. By developing $d\left(\boldsymbol{M}_t^{(-\nu)}\left(\boldsymbol{M}_t^{(-\nu)}\right)^{\dagger}\right)$ we get the following closed matrix SDE,
\begin{align*}
d\boldsymbol{Z}_t^{(-\nu)}=\frac{1}{\sqrt{2}}\sqrt{\boldsymbol{Z}_t^{(-\nu)}}d\boldsymbol{W}_t\sqrt{\boldsymbol{Z}_t^{(-\nu)}}+\frac{1}{\sqrt{2}}\sqrt{\boldsymbol{Z}_t^{(-\nu)}}d\boldsymbol{W}^{\dagger}_t\sqrt{\boldsymbol{Z}_t^{(-\nu)}}+(N-2\nu)\boldsymbol{Z}_t^{(-\nu)}dt, 
\end{align*}
for a complex matrix Brownian motion $\left(\boldsymbol{W}_t;t \ge 0\right)$. By Theorem 4 of \cite{MatrixYamadaWatanabe} the eigenvalue evolution $\left(\eta^{(-\nu)}_1(t) \le \cdots \le \eta^{(-\nu)}_N(t);t \ge 0 \right)$ of $\left(\boldsymbol{M}_t^{(-\nu)}\left(\boldsymbol{M}_t^{(-\nu)}\right)^{\dagger};t \ge 0\right)$, which form the squared singular values of $\left(\boldsymbol{M}_t^{(-\nu)};t \ge 0\right)$, satisfies,
\begin{align*}
d\eta_i^{(-\nu)}(t)=\sqrt{2}\eta_i^{(-\nu)}(t)d\beta_i(t)+\left[\left(N-2 \nu\right)\eta_i^{(-\nu)}(t)+\sum_{k \ne i}^{}\frac{2\eta_i^{(-\nu)}(t)\eta_k^{(-\nu)}(t)}{\eta_i^{(-\nu)}(t)-\eta_k^{(-\nu)}(t)}\right]dt, \ 1 \le i \le N,
\end{align*}
for some independent standard (real) Brownian motions $\{\beta_i\}_{i=1}^N$. Moreover, by making the change of variables $\delta_i^{(-\nu)}=\log\left(\eta_i^{(-\nu)}\right)$ we arrive at,
\begin{align*}
d\delta_i^{(-\nu)}(t)&=\sqrt{2}d\beta_i(t)+\left[\left(N-1-2 \nu\right)+\sum_{k \ne i}^{}\frac{2e^{\delta_k^{(-\nu)}(t)}}{e^{\delta_i^{(-\nu)}(t)}-e^{\delta_k^{(-\nu)}(t)}}\right]dt\\
&=\sqrt{2}d\beta_i(t)+\left[-2\nu+\sum_{k \ne i}^{}\frac{e^{\delta_i^{(-\nu)}(t)}+e^{\delta_k^{(-\nu)}(t)}}{e^{\delta_i^{(-\nu)}(t)}-e^{\delta_k^{(-\nu)}(t)}}\right]dt\ 1 \le i \le N.
\end{align*}
As in the proof of Lemma 11 of \cite{RiderValko}, we observe the following: First,
\begin{align*}
\sum_{k \ne 1}^{}\frac{e^{\delta_1^{(-\nu)}(t)}+e^{\delta_k^{(-\nu)}(t)}}{e^{\delta_1^{(-\nu)}(t)}-e^{\delta_k^{(-\nu)}(t)}}\le 1-N
\end{align*}
and furthermore, that changing $i$ to $i+1$ the interaction term changes by at most,
\begin{align*}
2\frac{e^{\delta_{i+1}^{(-\nu)}(t)}+e^{\delta_{i}^{(-\nu)}(t)}}{e^{\delta_{i+1}^{(-\nu)}(t)}-e^{\delta_i^{(-\nu)}(t)}}.
\end{align*}
Thus, for $i=1, \dots, N-1$, the difference $\delta_{i+1}^{(-\nu)}-\delta_i^{(-\nu)}$ is bounded above by the solution of,
\begin{align*}
dy_i(t)=\sqrt{2}\left(d\beta_{i+1}(t)-d\beta_i(t)\right)+2\left(\frac{1+e^{-y_i(t)}}{1-e^{-y_{i}(t)}}\right)dt
\end{align*}
and similarly, $\delta_1^{(-\nu)}$ by the solution of,
\begin{align*}
d\tilde{\delta}_1^{(-\nu)}(t)=\sqrt{2}d\beta_1(t)+(-2\nu+1-N)dt.
\end{align*}
Hence,
\begin{align*}
\lim_{t \to \infty}\frac{\delta_N^{(-\nu)}(t)}{t}\le \lim_{t \to \infty}\frac{\tilde{\delta}_1^{(-\nu)}(t)}{t}+\sum_{i=1}^{N-1}\lim_{t \to \infty}\frac{y_i(t)}{t}
= (-2 \nu+1-N)+2(N-1) \textnormal{ almost surely.}
\end{align*}
\end{proof}

\bigskip
\noindent
{\sc Mathematics Institute, University of Warwick, Coventry CV4 7AL, U.K.}\newline
\href{mailto:T.Assiotis@warwick.ac.uk}{\small T.Assiotis@warwick.ac.uk}

\end{document}